\documentclass[11pt, a4paper]{amsart}
\usepackage{amsmath,amssymb,amscd,amsfonts,graphicx,verbatim}

\def \al {\alpha }
\def \b {\beta}

\def\na  {\nabla}

\def \La {\Lambda}

\def\sseteq {\subseteq}
\def\mk{\medskip}

\def\s{\sigma}
\def\mk{\medskip}

\def\g {\gamma}

\def\P {{\mathcal P}}

\def\na {\nabla}

\def\mk {\medskip}
\def\bk {\bigskip}

\def\T {\mathcal{T}}

\def\D {\Delta}

\def\H {\mathcal{H}}
\def\B {\mathcal{B}}
\newcommand{\R}{{\mathbb R}}

\newcommand{\N}{{\mathbb N}}

\newtheorem{teo}{Theorem }[section]
\newtheorem{prop}[teo]{Proposition }
\newtheorem{rem}[teo]{Remark }

\newtheorem{lem}[teo]{Lemma }
\newtheorem{cor}[teo]{Corollary }
\newtheorem{defi}[teo]{Definition }

\begin{document}

\title{Tangential approximation of analytic sets }


\author{M. Ferrarotti} 
\address{Dipartimento di Matematica\\Politecnico di Torino\\Corso Duca
degli Abruzzi 24\\I-10129 Torino, Italy}
\email{ferrarotti@polito.it}
\thanks{This research was partially supported by M.I.U.R. and by G.N.S.A.G.A}

\author{E. Fortuna}
\address{Dipartimento di Matematica\\Universit\`a di Pisa\\Largo
  B. Pontecorvo 5\\I-56127 Pisa, Italy}
\email{elisabetta.fortuna@unipi.it}

\author{L. Wilson}
\address{Department of Mathematics\\University of Hawaii, Manoa\\Honolulu, HI 96822, USA}
\email{les@math.hawaii.edu}


\date{May 12, 2019}

\begin{abstract}

Two subanalytic subsets of $\R^n$ are called $s$-equivalent at a common point $P$
if the Hausdorff distance between their intersections with the sphere centered at $P$ of
radius $r$ vanishes to order $>s$ as $r$ tends to $0$.
 In this work we strengthen this notion in the case of real subanalytic subsets of $\R^n$
with isolated singular points, introducing the
notion of tangential $s$-equivalence at a common singular point which considers also the distance
between the tangent planes to the sets near the point. We prove that, if $V(f)$ is the zero-set
of an analytic map $f$ and if we assume that $V(f)$ has an isolated singularity, 
 say at the origin $O$, then for any $s\geq 1$
the truncation of the Taylor series of $f$ of sufficiently high order defines an algebraic set
with isolated singularity at $O$ which is tangentially $s$-equivalent to $V(f)$.

\end{abstract}

\maketitle 

\section{Introduction}\label{intro}

If $A$ and $B$ are two closed subanalytic subsets of $\R^n$, the Hausdorff distance 
between their intersections with the sphere of radius $r$ centered 
at a common point $P$ can be used  to ``measure'' how 
near the two sets  are at $P$.  We say that $A$ and $B$ are 
$s$--equivalent (at $P$) if the previous  distance tends to $0$ more rapidly than $r^s$ 
(if so, we write $A \sim_sB$).

In the papers  \cite{FFW-SNS}, \cite{FFW-germs}, \cite{FFW-semi} and  \cite{FFW-dim} we addressed 
the question of the existence of an algebraic representative $Y$ in the class of 
$s$--equivalence of a given subanalytic set $A$ at a fixed point $P$. In this case we also 
say that $Y$ $s$-approximates $A$.

The answer to the previous  question is in general negative for subanalytic sets. However,  for any real number $s\geq 1$ and for any closed semianalytic set $A \subset \R^n$ of
codimension $\geq 1$, there exists an algebraic subset
$Y$ of $\R^n$ such that $A \sim_sY$ (and $Y$ can be chosen of the same dimension as $A$; see  \cite{FFW-semi} and  \cite{FFW-dim}). 

In this paper we define a similar but stronger and geometrically significant equivalence relation: we say that two subanalytic sets $A$ and $B$ having an isolated singularity at $P$ are tangentially $s$-equivalent if  not only the points but also the tangent spaces to $A$ and $B$ are sufficiently ``close" of order $s$ near $P$.

If $V$ is the zero set of an analytic map $f:\R^n\rightarrow \R^p$ such that $f(P)=0$ and if $V$ is the closure of its regular points, then in  \cite{FFW-germs} we showed that for $k\in \N$ sufficiently large $V\sim_s V^k$, where $V^k$ is the zero set of the $k$-th Taylor polynomial of $f$ at $P$;
in fact the same is true for any representative of the $k$-jet of $f$ at $P$. 

In the present work we prove in Theorem \ref{main}  that, for any analytic map $f$ defining an analytic set $V$ with an isolated singularity at $O$ and for any $s$, 
the truncation of the Taylor series of $f$ of sufficiently high order defines an algebraic set
with isolated singularity at $O$ which is tangentially $s$-equivalent to $V$.

\section{Basic notions and notation}\label{basic}

In this section we recall the definition and
some basic properties of $s$-equivalence of subanalytic sets at a common point which,
without loss of generality, we can assume to be the origin $O$ of $\R^n$.

If $A$, $B$ are non-empty compact subsets of $\R^n$,
let $\delta (A,B) = \sup _{x\in B} d(x,A) $, where $d(x,A) =\inf _{y\in A}\|x-y\|$. Thus, if we denote  by $D(A,B)$
the classical Hausdorff distance between the two sets, we have that
$$D(A,B)= \max  \{\delta (A,B),\ \delta (B,A)\}.$$

\begin{defi} Let $A$ and $B$ be closed subanalytic subsets
of $\R^n$ with $O\in A\cap B$. Let $s$ be a real number $\geq 1$. Denote by $S_r$ the
sphere of radius $r$ centered at the origin.
\begin{enumerate}
\item[(a)]  We say that $A\leq_s B$  if one of the following conditions holds:
\begin{enumerate}
\item[(i)] $O$ is isolated  in $A$,
\item[(ii)] $O$ is non-isolated both in $A$ and in $B$ and
$$\lim_{r\to 0}\frac{\delta(B \cap S_r,A\cap S_r)}{r^s} =0.$$
\end {enumerate}
\item[( b)] We say that $A$ and
$B$ are $s$--equivalent  (and we will write $A \sim_sB$) if
$A\leq_s B$ and $B\leq_s A$.
\end{enumerate}
\end{defi}

It is easy to check that $\leq_s$ is transitive and that $\sim_s$ is an
equivalence relation.

Let  $B(O,R)$ denote the open ball centered at $O$ of radius $R$.
Observe that, if there exists $R>0$ such that $A\cap B(O,R) \subseteq B$,
then $A \leq_sB$ for any $s\geq 1$.
\smallskip

A useful tool to test the $s$-equivalence of two subanalytic sets is introduced in
the following definition:

\begin{defi}  Let $A$ be  a closed  subanalytic subset
of $\R^n$, $O\in A$. For any real positive constant $ \sigma$,
we will call  \emph{horn-neighbourhood} with center $A$ and exponent $ \sigma$ the set
$$\mathcal H(A,\sigma)= \{ x\in \R^n \ |\  d(x,A)<  \|x\|^\sigma\}.$$
\end{defi}

Observe that, if $\sigma \geq 1$, then 
$$\mathcal H(A,\sigma)\cap B(O,1)= \{ x\in B(O,1) \ |\  \exists y \in A\setminus \{O\}
\mbox{ such that } \|x-y\|<  \|x\|^\sigma\}.$$

\begin{prop}\label{horn-lemma} (\cite{FFW-germs}) Let $A,B$ be closed subanalytic subsets
of $\R^n$ with $O\in A\cap B$ and let $s \geq 1$. Then
$A\leq_s B$ if and only if there exist real constants $R>0$ and $\sigma >s$
such that $$(A \setminus \{O\}) \cap B(O,R)\subseteq
\mathcal H(B,\sigma).$$
\end{prop}

Now, restricting to the case when $O$ is an isolated singularity,  we are going to strengthen the notion of $s$-equivalence imposing additional conditions of differentiable nature.

\begin{defi}\label{IS} Let $A\subseteq \R^n$ be a closed subanalytic set.  
We say that $A$ is  an \emph{isolated singularity subanalytic set} (for short \emph{IS}) of dimension $d$ if the origin $O$ is a non-isolated point of $A$ and $A\setminus\{O\}$ is a submanifold of pure dimension $d$.
\end{defi}

If $A$ is an IS and $x\in A\setminus\{O\}$, let $T_xA$ be the tangent affine subspace to $A\setminus\{O\}$ at $x$.
If we denote with $\mathcal{P}_{n,d}$ the set of affine $d$-dimensional subspaces in $\R^n$, the tangent bundle of $A$ will be the subset
of $\R^n\times\P_{n,d}$ given by
$$\mathcal{T}A = \{(x,T_xA)\ |\ x\in A\setminus\{O\}\}.$$

We define a``distance" between affine subspaces as follows. A \emph{unit basis} of a vector subspace of $\R^n$ is a basis whose elements are unit vectors.
If $T_1,\ T_2\in\P_{n,d}$, for $i=1,2$ let $\nu(T_i)$ be the vector subspace orthogonal to the direction of $T_i$. If $\B_i=\{v^i_1,\dots ,v^i_{n-d}\}$ is a unit basis of $\nu(T_i)$,
let $\D(\B_1,\B_2)=max\{\| v^1_j-v^2_j\| \ |\ j=1,\dots , n-d\}$. Then we set $\D(T_1,T_2)=\inf \D(\B_1,\B_2)$, where $\B_i$ varies among the unit bases of $\nu(T_i)$.


Evidently $\D(T_1,T_2)=0$ if and only if $T_1$ and $T_2$ are parallel.

\begin{defi}\label{tangential-horn} Let $A\subseteq \R^n$ be an IS of dimension $d$ . 
For any real positive constant $\tau$
we call \emph{tangential horn neighbourhood} with center $A$ and exponent $\tau$ the set
$$
\begin{array}{rl}
\T\H(A,\tau) = \{(x,T)\in \R^n\times \P_{n,d}\  |\ & x\in T,\   \exists y\in A\setminus\{O\} \ \textrm{such that}\ \\
& \|x-y\|<\| x\|^\tau,\ \D(T_yA,T)<\|x\|^\tau\}.
\end{array}$$
\end{defi}

For any positive real number $R$ we set 
$$\mathcal{T}_R A = \{(x,T_xA)\ |\ x\in A\setminus\{O\}\cap B(O,R)\}.$$

\begin{defi}\label{super-s} Let $A, B \subseteq \R^n$ be two IS's of the same dimension $d$. If $s\geq 1$, we say that $A$ and $B$ are \emph{tangentially $s$-equivalent}
($A\approx_s B$) if there exist real constants $R>0$ and $\tau>s$ such that $\T_R A\subseteq \T\H(B,\ \tau)$ and
$\T_R B\subseteq \T\H(A,\ \tau)$.
\end{defi}

It is easy to see that, if $A\approx_s B$, then $A\sim_s B$. Moreover, as an immediate consequence of Definition \ref{super-s}, we have that if $A$ and $B$ are tangentially 1-equivalent then they have the same Nash fiber at $O$.

\section{Analytic maps and truncations}

In this section we collect some results that will be used in the final section to prove our main theorem. Some of these propositions are modified versions of results already used in 
\cite{FFW-germs} and \cite{FFW-semi}.
 
If $f\colon\R^n\to \R^p$ is an analytic map, let $V(f)=\{ x\in\R^n\ |\ f(x)=0\}$ denote its
zero-set.

An essential tool that we will repeatedly use in compact neighbourhoods of $O$ is the 
following slightly modified version of the classical \L ojasiewicz' inequality:

\begin{prop}\label{Loj} Let $A$ be a compact subanalytic subset of
$\R^n$. Assume $\varphi$ and $\psi$ are subanalytic functions defined on $A$
such that $\varphi$ is continuous, $V(\varphi) \subseteq V(\psi)$, $\psi$  is continuous at
the points of $V(\psi)$ and such that $\sup |\psi|<1$. Then there exists
a positive constant $\alpha$ such that $|\psi|^\alpha \leq |\varphi|$ on $A$ and $|\psi|^\alpha <
|\varphi|$ on $A\setminus V(\varphi)$.
\end{prop}

As a consequence of the previous \L ojasiewicz' inequality we get:

\begin{lem}\label{3} 
If  $f=(f_1,\dots ,f_p)\colon\R^n\to \R^p$ is an analytic map, then
\begin{enumerate}
\item there exists $\al\in\R^+$ such that
$\|f(x)\|>d(x,V(f))^\al$ for all $x\in\R^n\setminus V(f)$ near enough to $O$,
\item  there exists $\g\in\R^+$ such that $\| \na f_i(x) - \na f_i(y)\|<\| x-y\|^\g$ for all $i=1,\dots , p$ and for all
$x,\ y\in\R^n$ near enough to $O$ and with $x\neq y$.
\end{enumerate}
\end{lem}

\begin{proof} (1) It is a straightforward consequence of Proposition \ref{Loj}.\\
(2) By Proposition \ref{Loj}, for any $i$ there is $\g_i>0$ such that  $\| \na f_i(x) - \na f_i(y)\|
<\| x-y\|^{\g_i}$ for all
$x,\ y\in\R^n$ near enough to $O$ and with $x\neq y$.
Then it is enough to take $\g=\min\{\g_i | \ i=1,\dots ,p\}$.
\end{proof}

\begin{rem}\label{rem2} {\rm If $\g$ fulfills Lemma \ref{3}  and if $\g'<\g$, then $\g'$ has the same property.}
\end{rem}

\begin{lem}\label{1+2} Let $A$ be a closed subanalytic set such that $O$ is a non-isolated point of $A$ and let $\varphi \colon\R^n\to \R$ be a continuous subanalytic function such that $\varphi(x)>0$ if $x\in A\setminus\{O\}$.  Then there exist $\b,\ \s\in\R^+$ such that
$\varphi(y) >\|x\|^\b$ for any $x\in A\setminus\{O\}$ near enough to $O$, and for any $y\in B(x,\| x\|^\s)$.\end{lem}

\begin{proof} Let $A_0=A\setminus\{O\}$. By Proposition \ref{Loj}, there exists  $\beta >0$
such that $\varphi(x)> \|x\|^\beta$ for any $x\in A_0$ near $O$.

Consider the closed subanalytic set
$$W=\{(x,y)\in A\times \R^n \ |\ \varphi(y)\leq \|x\|^\beta\}.$$

The function $\delta\colon A\to \R$ defined by $\delta(x)= d((x,x),W)$ is subanalytic and continuous on $A$ and positive on $A_0$. Then, again
by Proposition \ref{Loj}, there exists $\s >0$
such that $\delta(x) > \|x\|^\s$ on $A_0$ near $O$.
Hence $(x,y)\notin W$ for any $x\in A_0$ near $O$ and for any $y \in B(x,\|x\|^\s)$: otherwise
there exist a sequence $x_i\in A_0$ converging to $O$ and a sequence $y_i\in  B(x_i,\|x_i\|^{\sigma})$ such that $(x_i, y_i)\in W$. Then
$$\delta(x_i)=d((x_i,x_i),W)\leq \|(x_i,y_i)-(x_i,x_i)\|+d((x_i,y_i),W)=\|x_i-y_i\|<\|x_i\|^\s,$$
a contradiction.

Then for any $x \in A_0$ near $O$ and for any $y \in B(x, \|x\|^\s)$ we have $\varphi(y)>\|x\|^\b.$
\end{proof}

\begin{rem}\label{rem1} {\rm Assume that $\b$ and $\s$ fulfill Lemma \ref{1+2}  and that $\b'>\b$ and $\s'>\s$.  Then $\b'$ and $\s'$ have the same property.}
\end{rem}

Let $\phi\colon\R^n\to\R^p$ be an analytic map and denote by $d_x \phi$ the differential of $\phi$ at $x$. Following \cite{Trotman-Wilson}, consider the function on $\R^n$ defined by
$$\La\phi(x)=\begin{cases} 0 &\text{if}\ rk(d_x\phi)<p\\
\inf\{\|d_x\phi(v)\|\ |\ v\perp\ker(d_x\phi),\|v\|=1\} &\text{if}\ rk(d_x\phi)=p\end{cases}.$$

It can be checked that $\La\phi$ is continuous and subanalytic. \mk

As usual we endow
$\operatorname {Hom} (\R^n,\R^p)$ with the standard norm
$$\|L\|= \max_{u\ne 0}
\frac{\|L(u)\|}{\|u\|}$$
for any linear map $L:\R^n\to\R^p$.

The next proposition is a direct consequence of Proposition 3.3 in \cite{FFW-germs}.

\begin{prop}\label{cont}
Let $\phi$ and $\phi'$ be analytic maps from $\R^n$ to $\R^p$.
If there is a positive function $\epsilon(x)$ such that $\|d_x\phi - d_x\phi'\| \leq \epsilon(x)$ for any $x$, then
$|\La\phi(x)-\La\phi'(x)|\le \epsilon(x)$ for any $x$.
\end{prop}
\mk

Another useful result  we will need is the following:

%

\begin{lem}\label{(3)} \cite[Lemma 3.5]{FFW-germs} 
Let $\phi\colon\R^n \rightarrow \R^p$ be
an analytic map which is a submersion
on an open ball $B(x,\rho)$. Let $r>0$ and assume that
$\Lambda\phi(y)\ge \frac r \rho$ for all $y\in B(x,\rho)$. Then
$\phi(B(x,\rho))\supseteq B(\phi(x),r)$.
\end{lem}

If $f\colon\R^n\to \R^p$ is an analytic map and $k\in \N$,
we will denote by $T^kf(x)$
the polynomial map whose components are the Taylor polynomials of
order $k$ at $O$ of the components of $f$. Moreover
we set $V=V(f)$, $V_0= V\setminus \{O\}$ and $V^k=V(T^kf)$.

\begin{defi}\label{fIS} If $f\colon\R^n\to \R^p$ is an analytic map, we say that $f$ defines  an \emph{isolated singularity analytic set} (for short, \emph{$f$ defines an IS}) if the origin $O$ is a non-isolated point of $V(f)$ and $f$ is submersive on $V(f)\setminus\{O\}$.
\end{defi}

Evidently, if $f\colon\R^n\to \R^p$ defines an IS, the set $V(f)\setminus\{O\}$ is an analytic submanifold of $\R^n$ of dimension $d=n-p$, i.e. $V(f)$ is an IS of dimension $d=n-p$. 

Observe also that, if $f\colon\R^n\to\R^p$ is an analytic map, then $f$ defines an IS if and only if $O$ is a non-isolated point of $V(f)$ and $\La f$ is positive on $V(f)\setminus \{O\}$.
\smallskip

When $f$ defines an IS, we already know (\cite{FFW-germs}, Corollary 4.2)  that 
$V(f)$ can be approximated of order $s$ by the zero-set of a suitable truncation of the Taylor series of $f$. In the next section we will strengthen this result obtaining a tangential approximation. To do that, we will use two results which are particular cases of arguments used in \cite{FFW-germs}. Since these results do not appear as independent statements in that paper, we conclude this section presenting them with their proofs  for the sake of the reader.

\begin{prop}\label{7} Let $f\colon\R^n\to \R^p$ be an analytic map.
If $\al$  fulfills the thesis of Lemma \ref{3}, then there exists a real constant  $R>0$ such that for any $\s>0$ and for any $k>\al\s$ we have
$$(V^k\setminus \{O\}) \cap B(O,R) \subseteq \H(V,\s).$$
\end{prop}

\begin{proof}
By Lemma \ref{3}, we have that $\|f(x)\|> d(x,V)^\alpha $ for all $x \in \R^n \setminus V$ 
near $O$. Then for $x \not\in
\mathcal H(V,\sigma)$ we have that $\|f(x)\| > \|x\|^{\alpha\sigma}$.

Let $k$ be an integer such that $k > \alpha\sigma$. Then
$$ \lim_{x\to O}\frac{\|f(x)-T^kf(x)\|}{\|x\|^{\alpha\sigma}}=0.$$
It follows  that $V^k\setminus \{O\} \subseteq \mathcal H(V,\sigma)$ near $O$: otherwise
there would exist a sequence of
points $y_i\ne O$ converging to $O$ such that $y_i\in V^k\setminus\mathcal H(V,\sigma)$ and hence
$$ \lim_{i \to \infty}\frac{\|f(y_i)-T^kf(y_i)\|}{\|y_i\|^{\alpha\sigma}}=
\lim_{i \to \infty}\frac{\|f(y_i)\|}{\|y_i\|^{\alpha\sigma}} \geq 1$$
which is a contradiction.
\end{proof}

\begin{prop}\label{5} Let $f\colon\R^n\to \R^p$  be an analytic map which defines an IS. Assume that $\b$ and $\s$ are exponents which satisfy  Lemma \ref{1+2} when we take $A=V(f)$ and $\varphi=\La f$ and let $\s >1$. Then 
\begin{enumerate}
\item[(i)] if $k$ is an integer such that $k>\b+1$, then $\La T^kf(y)>\|x\|^{\b+1}$ for any $x\in V_0$ near $O$ and for any 
$y\in B(x,\|x\|^\s)$; in particular $T^kf$ is a submersion on $B(x,\|x\|^\s)$,
\item[(ii)] there exists a real constant $R>0$ such that 
$$V_0 \cap B(O,R) \subseteq \H(V^k,\s)$$
for all integers $k$ such that $k>\b+\s+1$. In particular $O$ is not isolated in $V^k$.
\end{enumerate}
\end{prop}

\begin{proof}
(i) Assume for contradiction that there exist a sequence
$x_i\in V_0$ converging to $O$ and a sequence $y_i\in  B(x_i,\|x_i\|^{\sigma})$
such that $\La T^kf(y_i) \leq  \|x_i\|^{\beta +1}$. Thus by Lemma \ref{1+2}  we
have
$$\frac{\La f(y_i) - \La T^kf(y_i)}{\|x_i\|^\beta}
> \frac{\|x_i\|^\beta - \|x_i\|^{\beta +1}}{\|x_i\|^\beta}= 1 - \|x_i\|.$$

On the other hand, by Taylor expansion and Proposition \ref{cont}

$$0 
\leq \frac {\La f(y_i) - \La T^kf(y_i)}{\|x_i\|^\beta}
\leq \frac{\|y_i\|^{k-1}}{\|x_i\|^\beta} 
\leq \frac{(\|y_i-x_i\|+\|x_i\|)^{k-1}}{\|x_i\|^\beta} =$$
$$=\left(\frac{\|y_i-x_i\|}{\|x_i\|^h}+\|x_i\|^{1-h}\right)^{k-1} <
\left( \|x_i\|^{\sigma-h}+\|x_i\|^{1-h}\right)^{k-1}
$$
where $h=\frac{\beta}{k-1}$. Since $\sigma> 1$ and $h<1$, we have that
$$\frac{\La f(y_i) - \La T^kf(y_i)}{\|x_i\|^\beta}$$
converges to $0$, which is a contradiction.

(ii)
Let $k$ be an integer such that $k>\b+\s+1$. Using (i)  we have that, 
if $x\in V_0$  near enough to $O$, then $\La T^kf(y)>\|x\|^{\b+1}$ for any $y\in B(x,\|x\|^\s)$.
So we can apply  Lemma \ref{(3)}  with
$\phi=T^kf$, $r=\|x\|^{\b+\s +1}$ and $\rho= \|x\|^\sigma$, obtaining that
$$T^kf(B(x,\|x\|^\sigma))\supseteq B(T^kf(x), \|x\|^{\b+\s+1}).$$
Moreover  we have
$$\lim_{\begin{array}{c}
\scriptstyle z \to O \\ \scriptstyle z\in V_0 \end{array}}
\frac{\|T^kf(z)\|}{\|z\|^k}= \lim_{\begin{array}{c}
\scriptstyle z \to O \\ \scriptstyle z\in V_0 \end{array}}
\frac{\|T^kf(z)-f(z)\|}{\|z\|^k}=0.$$
As a consequence if $x\in V_0$ is sufficiently near $O$, then $\|T^kf(x)\|<\|x\|^k <
\|x\|^{\b+\s+1}$ and hence $O$ belongs to 
$B(T^kf(x), \|x\|^{\b+\s+1})$; so there exists
$y\in B(x,\|x\|^{\sigma})$ such that $T^kf(y)=O$. Then near $O$ we have that $V_0 \subseteq \H(V^k,\sigma)$
and our thesis is proved.
\end{proof}

As an immediate consequence of Proposition \ref  {7}   and  Proposition \ref{5}  we obtain 

\begin{cor}\label{week-equiv} Let $f\colon\R^n\to \R^p$  be an analytic map which defines an IS. Then  $V(f)\sim_s V(T^k f)$ for $k$ sufficienty large.
\end{cor}

\section{Main theorem}\label{main-proof}

We are now ready to prove our main result.

\begin{teo}\label{main}  Let $f\colon\R^n\to \R^p$  be an analytic map which defines an IS. 
If $s\geq 1$, there exists $k_0\in \N$ such that for all integers $k\geq k_0$ the map $T^kf$ defines an IS and $V(f)\approx_s V(T^kf)$.
\end{teo}

\begin{proof} As in the previous section, set $f=(f_1,\dots ,f_p)$, $V=V(f)$, $V_0= V\setminus \{O\}$ and $V^k=V(T^kf)$.

At first let us prove that we can find $k_0\in\N$ such that for any integer $k\geq k_0$ the map $T^kf$ defines an IS and there exist $R>0$ and  $\tau > s$ such that  $\T_R V\subseteq \T\H( V^k,\tau)$.

Take $\al, \g$ so that they fulfill the thesis of Lemma \ref{3}. Since by hypothesis $f$ defines an IS, then $\La f$ is positive on $V_0$, so let $\b_0, \s_0$ be exponents which satisfy  Lemma \ref{1+2} when we take $A=V(f)$ and $\varphi=\La f$.

Moreover, for $i=1,\dots ,p$ apply Lemma \ref{1+2} with $A=V(f)$ and $\varphi =\| \nabla f_i\|$ to get exponents $\b_i$ and $\s_i$. If we set $\b=\max\{\b_i\ |\ i=0,\dots, p\}$ and
$\s=\max\{\s_i\ |\ i=0,\dots, p\}$,
by Remark \ref{rem2} and Remark \ref{rem1} we can assume that $\g \leq 1$ and
$\s>\frac{\b+s}{\g}>1$.

At first let us prove that there exists an integer $k_0$ such that $T^kf$ defines an IS for all $k\geq k_0$.
Namely, if we consider $U=\bigcup_{x\in V} B(x,\|x\|^\sigma)$, the sets $V$ and $W=\R^n \setminus U$ are subanalytic and 
meet only at $O$, so they are regularly situated, 
i.e. there exists $\mu$ such that $d(x,V) + d(x, W) > \|x\|^{\mu}$ 
for all $x$ near $O$. In particular if $x\in \H(V,\mu)$ then $d(x,W)>0$ and hence $x\in U$, 
i.e. $\H(V,\mu)\ \subseteq U$. 

Let $k_0$ be the integer part of $\max\{\al\s,\b+\s+1, \alpha\mu\}+1$.

Then for any $k \geq k_0$ by Proposition \ref{5} (ii), $O$ is not isolated in $V^k$; moreover, since $k > \alpha \mu$, by Proposition \ref{7}, $V^ k\setminus \{O\}\subseteq \H(V,\mu)\subseteq U$ near $O$. 
This implies that for any $y\in V^ k\setminus \{O\}$ there exists $x\in V$ such that 
$y\in B(x, \|x\|^\s)$; thus   by Proposition \ref{5} (i)
we have that $\La T^kf$ does not vanish on $V^ k\setminus \{O\}$, and therefore the map $T^kf$ defines an IS. 
\smallskip

Moreover for any $k \geq k_0$, again  by Proposition \ref{5} (ii), we have that $V_0 \cap B(O,R) \sseteq \H(V^k,\s)$ for some $R$; so, for each $x\in V_0\cap B(O,R) $ there is $y\in V^k \setminus \{O\}$ such that
$\|x-y\|<\|x\|^\s$, and hence $\|y\| < \|x\| +\|x\|^\s$. Then $y\in B(x, \|x\|^\sigma)$ and, since 
$k>\beta +1$, by 
Proposition \ref{5} (i) $T^kf$ is submersive at $y$. For such $x,y$ let us estimate $\D(T_xV,T_yV^k)$. 

The vector spaces $\nu(T_xV)$ and $\nu(T_yV^k)$ have bases $B_x=\{ \na f_i(x)\ |\ i=1, \ldots, p\}$ and $B^k_y=\{\na T^kf_i(y)\ |\ i=1, \ldots, p\}$ respectively; then, by Lemma \ref{3} and Taylor expansion,  we have for $i=1, \dots ,p$
$$\| \na f_i(x)-\na T^kf_i(y)\|\leq \| \na f_i(x)- \na f_i(y)\|+\| \na f_i(y)- \na T^kf_i(y)\|\leq \| x-y\|^\g +  \| y\|^{k-1}.$$

Since $\gamma \leq 1$ and $k \geq k_0 > \sigma +1$, then 
$k- \g\s-1>0$; thus  from the previous inequalities we get that near $O$
$$\|\na f_i(x)-\na T^kf_i(y)\|< \|x\|^{\g\s} + (\|x\| +\|x\|^\s)^{k-1}=$$
$$= \|x\|^{\g\s} \left( 1+ \|x\| ^{k- \g\s-1} ( 1 +\|x\| ^{\s -1})^{k-1}\right)\leq
2\|x\|^{\g\s}.$$

Since $\b+s<\g\s$, we can choose $\eta$ such that $\b+s<\eta<\g\s$; so  we can assume that
$$\| \na f_i(x)-\na T^kf_i(y)\|< \frac{1}{2}\|x\|^\eta.$$

We want to estimate $\Delta(\B_x,\B^k_y)$ where $\B_x$ and $\B^k_y$ are the unit bases obtained from $B_x$ and $B^k_y$ respectively by normalizing their elements. 

Observe that, if $u,\ v\in\R^n\setminus\{ O\}$, 
$$\left\| \frac{u}{\| u\|} - \frac{v}{\| v\|} \right\| =
\left\| \frac{\|v\|u - \| u\|v}  {\|u\|\| v\|} \right\| =
 \left\| \frac{\|v\|u -\|v\|v + \|v\|v - \| u\|v}  {\|u\|\| v\|} \right\| \leq 
\frac {   \|u-v\| +  \Big| \|v\|-\|u\|   \Big|} {  \|u\|} 
$$
hence the following inequality holds:
\mk
\begin{equation}\label{eqn:norme}
\left\| \frac{u}{\| u\|} - \frac{v}{\| v\|} \right\| \leq 2\frac{\| u-v\|}{\| u\|}.
\end{equation}
\mk

Applying inequality (\ref{eqn:norme}) and Lemma \ref{1+2} to $u=\na f_i(x)$ and $v=\na T^kf_i(y)$, we get

$$\left\| \frac{\na f_i(x)}{\| \na f_i(x)\|} - \frac{\na T^kf_i(y)}{\| \na T^kf_i(y)\|}\right\| 
<  \|x\|^{\eta-\b _i}\leq  \|x\|^{\eta-\b}.$$
\mk

Since $s<\eta-\b<\g\s-\b<\s$ and since $\|x-y\|<\|x\|^\s$, if we take $\tau=\eta-\b$ we get that
$\|x-y\| < \|x\|^\tau$  and $\Delta  (T_x V, T_y V^k) \leq \Delta (\B_x,\B^k_y)< \|x\|^\tau$. Hence 
$\T_R V\subseteq \T\H( V^k,\tau)$.
\bk

We  show now that, up to reducing $R$, we have that  $\T_R V^k\subseteq \T\H( V,\tau)$ with the same $k_0$ and $\tau$ as above.

By Proposition \ref{7}, if $k\geq k_0$, for any $x\in (V^k \setminus \{O\})\cap B(O,R)$ there is $y\in V_0$ such that
$\|x-y\|<\|x\|^\s < \|x\|^\tau$.
Since
$$\| \na f_i(y)-\na T^kf_i(x)\|\leq \| \na f_i(y)- \na f_i(x)\|+\| \na f_i(x)- \na T^kf_i(x)\|,$$
by  computations and arguments analogous to the ones used in the previous part
of our proof we can deduce that:

\begin{equation}\label{due}
\| \na f_i(y)-\na T^kf_i(x)\|
\leq\| x-y\|^\g +  \| x\|^{k-1}< \|x\|^{\g\s}+\|x\|^{k-1}.
\end{equation}

Since $y\in B(x,\|x\|^\s)$, we have $\|y\|>\|x\|-\|x\|^\s$. Moreover, since  $y\in V_0$, by Lemma \ref{1+2} we have that $\|\na f_i(y)\|>\|y\|^\b$.
 So from inequalities (\ref{eqn:norme}) and  (\ref{due}) we obtain that (near $O$)

$$\left\| \frac{\na f_i(y)}{\| \na f_i(y)\|} - \frac{\na T^kf_i(x)}{\| \na T^kf_i(x)\|}\right\|\leq 2\frac{\|\na f_i(y)-\na T^kf_i(x)\|}{\|\na f_i(y)\|}< 2\frac{\|x\|^{\g\s}+\|x\|^{k-1}}{\|y\|^\b}\leq $$

$$\leq  2\frac{\|x\|^{\g\s}+\|x\|^{k-1}}{(\|x\|-\|x\|^\s)^\b}=2\|x\|^{\g\s-\b}\frac{1+\|x\|^{k-\g\s-1}}{(1-\|x\|^{\s-1})^\b}\leq 3\|x\|^{\g\s-\b}\leq \|x\|^{\eta-\b}=\|x\|^\tau.$$
\mk

Hence $\T_R V^k\sseteq \T\H(V,\tau)$ and so $V\approx_s V^k$ for $k\geq k_0$.
\end{proof}

\begin{thebibliography}{10}


\bibitem[FFW1]{FFW-SNS}
{ M. Ferrarotti, E. Fortuna and  L. Wilson}:
Local approximation of semialgebraic sets.
{\em Ann. Sc. Norm. Super. Pisa Cl. Sci. (5)}, vol. I, n. 1
(2002), pp. 1--11


\bibitem[FFW2]{FFW-germs}
{M. Ferrarotti, E. Fortuna and  L. Wilson}:
Algebraic approximation of germs of real analytic sets.
{\em Proc. Amer. Math. Soc.}, vol. 138, n. 5
(2010), pp. 1537--1548

\bibitem[FFW3]{FFW-semi}
{ M. Ferrarotti, E. Fortuna and  L. Wilson}:
Local algebraic approximation of semianalytic sets.
{\em Proc. Amer. Math. Soc.}, vol. 143, n. 1 (2015), pp. 13--23

\bibitem[FFW4]{FFW-dim}
{ M. Ferrarotti, E. Fortuna and  L. Wilson}:
Algebraic approximation preserving dimension.
{\em Ann. Mat. Pura Appl. (4)}, vol. 196, n. 2 (2017), pp. 519--531

\bibitem[TW]{Trotman-Wilson}
{D. Trotman, L. Wilson}:
Stratifications and finite determinacy.
{\em Proc. London Math. Soc. (3)}, vol. 78 n. 2 (1999), pp. 334--368


\end{thebibliography}
\end{document}